\documentclass[12pt]{article}
\usepackage{amsmath}
\usepackage{amssymb}
\usepackage{amsthm}
\usepackage{extarrows}

\newtheorem{theorem}{Theorem}

\newtheorem{lemma}{Lemma}
\newtheorem{definition}{Definition}

\newtheorem{proposition}{Proposition}

\newcommand{\Mj}{\mathcal{M}_j}
\newcommand{\Mjj}{\mathcal{M}_{j-1}}
\newcommand{\Hj}{\mathcal{H}_j}
\newcommand{\Rm}{\mathbb{R}^m}
\newcommand{\C}{\mathbb{C}}
\newcommand{\Clm}{\mathcal{C}l_m}
\newcommand{\Ho}{\mathcal{H}_1}
\newcommand{\DD}{\mathcal{D}_{1,k}}
\newcommand{\Do}{\mathcal{D}_{1,2n-1}}
\newcommand{\De}{\mathcal{D}_{1,2n}}
\newcommand{\scs}{\mathcal{S}}
\newcommand{\Mo}{\mathcal{M}_1}

\begin{document}

\title{Higher Order Fermionic and Bosonic Operators}

\author{Chao Ding$^1$\thanks{These authors contributed equally to this manuscript.}\space\thanks{Electronic address:  {\tt dchao@uark.edu}.}, Raymond Walter$^{1,2 *}$\thanks{Electronic address:  {\tt rwalter@email.uark.edu}; R.W. acknowledges this material is based upon work supported by the National Science Foundation Graduate Research Fellowship Program under Grant No. DGE-0957325 and the University of Arkansas Graduate School Distinguished Doctoral Fellowship in Mathematics and Physics.}, and John Ryan$^1$\thanks{Electronic address: {\tt jryan@uark.edu.}} \\
\emph{\small $^1$Department of Mathematics, University of Arkansas, Fayetteville, AR 72701, USA} \\ 
\emph{\small $^2$Department of Physics, University of Arkansas, Fayetteville, AR 72701, USA}}

\date{}

\maketitle

\begin{abstract}
This paper studies a particular class of higher order conformally invariant differential operators and related integral operators acting on functions taking values in particular finite dimensional irreducible representations of the Spin group. The differential operators can be seen as a generalization to higher spin spaces of $k$th-powers of the Euclidean Dirac operator. To construct these operators, we use the framework of higher spin theory in Clifford analysis, in which irreducible representations of the Spin group are realized as polynomial spaces satisfying a particular system of differential equations. As a consequence, these operators act on functions taking values in the space of homogeneous harmonic or monogenic polynomials depending on the order. Moreover, we classify these operators in analogy with the quantization of angular momentum in quantum mechanics to unify the terminology used in studying higher order higher spin conformally invariant operators: for integer and half-integer spin, these are respectively bosonic and fermionic operators. Fundamental solutions and their conformal invariance are presented here.
\end{abstract}
{\bf Keywords:}\quad Higher order fermionic and bosonic operators, Conformal invariance, Fundamental solutions, Intertwining operators, Ellipticity.

\section{Introduction}\hspace*{\fill} 
\par Classical Clifford analysis started as a generalization of aspects of one variable complex analysis to $m$-dimensional Euclidean spaces. At the heart of this theory is the study of the Dirac operator $D_x$ on $\Rm$, a conformally invariant first order differential operator which generalizes the role of the Cauchy-Riemann operator. Moreover, this operator is related to the Laplace operator with $D_x^2=-\Delta_x$. The classical theory is centered around the study of functions on $\Rm$ and taking values in a spinor space \cite{At, Br}, and abundant results have been found. See for instance \cite{Br, De, P,G,R,R1}.\\
\par
P.A.M. Dirac constructed a first order relativistically covariant equation describing the dynamics of an electron by using Clifford modules; hence differential operators constructed using Clifford modules are called Dirac operators. In the presence of an electromagnetic field, the Dirac Hamiltonian for an electron acquires an additional contribution formally analogous to internal angular momentum called spin, from which the Spin group and related notions take their name; for the electron, spin has the value $\frac{1}{2}$ \cite{Lawson}. Indeed, in dimension four with appropriate signature, null-solutions of the Dirac operator $D_x$ from classical Clifford analysis correspond to solutions for the relativistically covariant dynamical equation of a massless particle of spin $\frac{1}{2}$, also called the Weyl equation. The Dirac equation for the electron, which has mass, may be considered an inhomogeneous equation satisfied by the Dirac operator $D_x$. The Dirac equation is not only relativistically covariant, but also conformally invariant. The construction of conformally invariant massless wave equations, in terms of invariant operators with conformal weights over spin fields, is well described by \cite{E}. The general importance of conformal invariance in physics has long been recognized \cite{RevModPhys}. \\
\par  

Rarita and Schwinger \cite{Ra} introduced a simplified formulation of the theory of particles of arbitrary half-integer spin $k+\frac{1}{2}$ and in particular considered its implications for particles of spin $\frac{3}{2}$. In the context of Clifford analysis, the so-called \emph{higher spin theory} was first introduced through the Rarita-Schwinger operator \cite{B}, which is named analogously to the Dirac operator and reproduces the wave equations for a massless particle of arbitrary half-integer spin in four dimensions with appropriate signature \cite{Ro}. (The solutions to these wave equations may not be physical \cite{VZ1, VZ2}.) The higher spin theory studies generalizations of classical Clifford analysis techniques to higher spin spaces \cite{B1, Br1, B, D, E, Li}. This theory concerns the study of the operators acting on functions on $\Rm$, taking values in arbitrary irreducible representations of $Spin(m)$. These arbitrary representations are defined in terms of polynomial spaces that satisfy certain differential equations, such as $j$-homogeneous monogenic polynomials (half-integer spin) or $j$-homogeneous harmonic polynomials (integer spin). More generally, one can consider the highest weight vector of the spin representation as a parameter \cite{HighestWeightExample}, but this is beyond our present scope. The present paper contributes to the study of conformally invariant operators in the higher spin theory.\\
\par
In principle, all conformally invariant differential operators on locally conformally flat manifolds in higher spin theory are classified by Slov\'{a}k \cite{J}; see also \cite{VS-ref}. This classification is non-constructive, showing only between which vector bundles these operators exist and what is their order; explicit expressions of these operators are still being found. Eelbode and Roels \cite{E} point out that the Laplace operator $\Delta_x$ is not conformally invariant anymore when it acts on $C^{\infty}(\Rm,\Ho)$, where $\Ho$ is the degree one homogeneous harmonic polynomial space (correspondingly $\Mo$ for monogenic polynomials). They construct a second order conformally invariant operator on $C^{\infty}(\Rm,\Ho)$, the (generalized) Maxwell operator. In dimension four with appropriate signature it reproduces the Maxwell equation, or the wave equation for a massless spin-1 particle (the massless Proca equation) \cite{E}. De Bie and his co-authors \cite{B1} generalize this Maxwell operator from $C^{\infty}(\Rm,\Ho)$ to $C^{\infty}(\Rm,\Hj)$ to provide the higher spin Laplace operators, the second order conformally invariant operators generalizing the Laplace operator to arbitrary integer spins. Their arguments also suggest that $D_x^k$ is not conformally invariant in the higher spin theory. This raises the following question: what operators generalize $k$th-powers of the Dirac operator in the higher spin theory? We know these operators exist, with even order operators taking values in homogeneous harmonic polynomial spaces and odd order operators in homogeneous monogenic polynomial spaces \cite{J}. This paper explicitly answers the question with the condition that the target space is a degree one homogeneous polynomial space, encompassing the spin-1 and spin-$\frac{3}{2}$ cases. More generally, one can consider bosonic and fermionic operators corresponding to either integer or half-integer spins, taking values in polynomial spaces of appropriate degree of homogeneity that are either harmonic or monogenic; however, their function theory is not fully examined here.\\
\par
The paper is organized as follows: We briefly introduce Clifford algebras, Clifford analysis, and representation theory of the Spin group in Section 2. In Section 3, we introduce the $k$-order higher spin operators $\DD$ as the generalization of $D_x^k$ when acting on $C^{\infty}(\Rm, U)$, where $U=\Ho$ (spin-1) or $U=\Mo$ (spin-$\frac{3}{2}$) depending on whether $k$ is even or odd. We overview classification, existence, and uniqueness results for higher spin operators. Nomenclature is given for the higher order higher spin operators that we consider: bosonic and fermionic operators. The construction and conformal invariance of the operators $\DD$ are given with the help of the concept of \emph{generalized symmetry}, as in \cite{B1,E}. Then we provide the intertwining operators for $\DD$, which also reveal that $\DD$ is conformally invariant. These intertwining operators are special cases of Knapp-Stein intertwining operators (\cite{CO,KS}) in higher spin theory. Section 4 presents the fundamental solutions  (up to a multiplicative constant) of $\DD$ with the help of Schur's Lemma from representation theory. We also present an argument that the fundamental solution to $\DD$ seen as a type of convolution operator is also conformally invariant. These convolution type operators can also be recovered as Knapp-Stein operators (\cite{CO,KS}) in higher spin theory. The expressions of the fundamental solutions also suggest that $\DD$ is a generalization of $D_x^k$. With the observation that the bases of the target spaces $U$ have simple expressions, we prove that $\DD$ is an elliptic operator in Section 5.

\section*{Acknowledgement}
The authors wish to thank an anonymous referee for helpful suggestions that significantly improved the manuscript. The authors are also grateful to Bent \O rsted for communications pointing out that the intertwining operators of our conformally invariant differential operators and  our convolution type operators can be recovered as Knapp-Stein intertwining operators  and Knapp-Stein operators in higher spin theory.

\section{Preliminaries}
\subsection{Clifford algebra}\hspace*{\fill}
A real Clifford algebra, $\Clm,$ can be generated from $\mathbb{R}^m$ by considering the
relationship $$\underline{x}^{2}=-\|\underline{x}\|^{2}$$ for each
$\underline{x}\in \mathbb{R}^m$.  We have $\mathbb{R}^m\subseteq \Clm$. If $\{e_1,\ldots, e_m\}$ is an orthonormal basis for $\mathbb{R}^m$, then $\underline{x}^{2}=-\|\underline{x}\|^{2}$ tells us that $$e_i e_j + e_j e_i= -2\delta_{ij},$$ where $\delta_{ij}$ is the Kronecker delta function. An arbitrary element of the basis of the Clifford algebra can be written as $e_A=e_{j_1}\cdots e_{j_r},$ where $A=\{j_1, \cdots, j_r\}\subset \{1, 2, \cdots, m\}$ and $1\leq j_1< j_2 < \cdots < j_r \leq m.$
Hence for any element $a\in \Clm$, we have $a=\sum_Aa_Ae_A,$ where $a_A\in \mathbb{R}$. Similarly, the complex Clifford algebra $\Clm (\C)$ is defined as the complexification of the real Clifford algebra
$$\Clm (\C)=\Clm\otimes\C.$$
We consider real Clifford algebra $\Clm$ throughout this subsection, but in the rest of the paper we consider the complex Clifford algebra $\Clm (\C)$ unless otherwise specified. \\
\par
The Pin and Spin groups play an important role in Clifford analysis. The Pin group can be defined as $$Pin(m)=\{a\in \mathcal{C}l_m: a=y_1y_2\dots y_{p},\ y_1,\dots,y_{p}\in\mathbb{S}^{m-1},p\in\mathbb{N}\},$$ 
where $\mathbb{S} ^{m-1}$ is the unit sphere in $\Rm$. $Pin(m)$ is clearly a group under multiplication in $\mathcal{C}l_m$. \\
\par
Now suppose that $a\in \mathbb{S}^{m-1}\subseteq \mathbb{R}^m$, if we consider $axa$, we may decompose
$$x=x_{a\parallel}+x_{a\perp},$$
where $x_{a\parallel}$ is the projection of $x$ onto $a$ and $x_{a\perp}$ is the rest, perpendicular to $a$. Hence $x_{a\parallel}$ is a scalar multiple of $a$ and we have
$$axa=ax_{a\parallel}a+ax_{a\perp}a=-x_{a\parallel}+x_{a\perp}.$$
So the action $axa$ describes a reflection of $x$ in the direction of $a$. By the Cartan-Dieudonn$\acute{e}$ Theorem each $O\in O(m)$ is the composition of a finite number of reflections. If $a=y_1\cdots y_p\in Pin(m),$ we define $\tilde{a}:=y_p\cdots y_1$ and observe that $ax\tilde{a}=O_a(x)$ for some $O_a\in O(m)$. Choosing $y_1,\ \dots,\ y_p$ arbitrarily in $\mathbb{S}^{m-1}$, we see that the group homomorphism
\begin{eqnarray}
\theta:\ Pin(m)\longrightarrow O(m)\ :\ a\mapsto O_a,
\end{eqnarray}
with $a=y_1\cdots y_p$ and $O_ax=ax\tilde{a}$ is surjective. Further $-ax(-\tilde{a})=ax\tilde{a}$, so $1,\ -1\in Ker(\theta)$. In fact $Ker(\theta)=\{1,\ -1\}$. See \cite{P1}. The Spin group is defined as
$$Spin(m)=\{a\in \mathcal{C}l_m: a=y_1y_2\dots y_{2p};\ y_1, \dots,y_{2p}\in\mathbb{S}^{m-1},p\in\mathbb{N}\}$$
 and it is a subgroup of $Pin(m)$. There is a group homomorphism
\begin{eqnarray*}
\theta:\ Spin(m)\longrightarrow SO(m)\ ,
\end{eqnarray*}
which is surjective with kernel $\{1,\ -1\}$. It is defined by $(1)$. Thus $Spin(m)$ is the double cover of $SO(m)$. See \cite{P1} for more details.\\
\par
For a domain $U$ in $\mathbb{R}^m$, a diffeomorphism $\phi: U\longrightarrow \mathbb{R}^m$ is said to be conformal if, for each $x\in U$ and each $\mathbf{v,w}\in TU_x$, the angle between $\mathbf{v}$ and $\mathbf{w}$ is preserved under the corresponding differential at $x$, $d\phi_x$.
For $m\geq 3$, a theorem of Liouville tells us the only conformal transformations are M\"obius transformations. Ahlfors and Vahlen show that given a M\"{o}bius transformation on $\mathbb{R}^m \cup \{\infty\}$ it can be expressed as $y=(ax+b)(cx+d)^{-1}$ where $a,\ b,\ c,\ d\in \Clm$ and satisfy the following conditions \cite{Ah}:
\begin{eqnarray*}
&&1.\ a,\ b,\ c,\ d\ are\ all\ products\ of\ vectors\ in\ \mathbb{R}^m;\\
&&2.\ a\tilde{b},\ c\tilde{d},\ \tilde{b}c,\ \tilde{d}a\in\mathbb{R}^m;\\
&&3.\ a\tilde{d}-b\tilde{c}=\pm 1.
\end{eqnarray*}
Since $y=(ax+b)(cx+d)^{-1}=ac^{-1}+(b-ac^{-1}d)(cx+d)^{-1}$, a conformal transformation can be decomposed as compositions of translation, dilation, reflection and inversion. This gives an \emph{Iwasawa decomposition} for M\"obius transformations. See \cite{Li} for more details.
\par
The Dirac operator in $\mathbb{R}^m$ is defined to be $$D_x:=\sum_{i=1}^{m}e_i\partial_{x_i}.$$  Note $D_x^2=-\Delta_x$, where $\Delta_x$ is the Laplacian in $\mathbb{R}^m$.  A $\Clm$-valued function $f(x)$ defined on a domain $U$ in $\Rm$ is left monogenic if $D_xf(x)=0.$ Since multiplication of Clifford numbers is not commutative in general, there is a similar definition for right monogenic functions. Sometimes we will consider the Dirac operator $D_u$ in vector $u$ rather than $x$.\\
\par
Let $\mathcal{M}_j$ denote the space of $\mathcal{C}l_m$-valued monogenic polynomials, homogeneous of degree $j$. Note that if $h_j(u)\in\Hj$, the space of $\mathcal{C}l_m$-valued harmonic polynomials homogeneous of degree $j$, then $D_u h_j(u)\in\mathcal{M}_{j-1}$, but $D_u up_{j-1}(u)=(-m-2j+2)p_{j-1}(u),$ so
$$\mathcal{H}_j=\mathcal{M}_j\oplus u\mathcal{M}_{j-1},\ h_j=p_j+up_{j-1}.$$
This is an \emph{Almansi-Fischer decomposition} of $\Hj$. See \cite{D} for more details. In this Almansi-Fischer decomposition, we define $P_j$ as the projection map 
\begin{eqnarray*}
P_j: \mathcal{H}_j\longrightarrow \mathcal{M}_j.
\end{eqnarray*}
Suppose again $U$ is a domain in $\mathbb{R}^m$. Consider a differentiable function $f: U\times \mathbb{R}^m\longrightarrow \mathcal{C}l_m,$
such that for each $x\in U$, $f(x,u)$ is a left monogenic polynomial homogeneous of degree $j$ in $u$, then the Rarita-Schwinger operator \cite{B,D} is defined by 
 $$R_jf(x,u):=P_jD_xf(x,u)=(\frac{uD_u}{m+2j-2}+1)D_xf(x,u).$$
 Though we have presented the Almansi-Fischer decomposition, the Dirac operator, and the Rarita-Schwinger operator here in terms of functions taking values in the real Clifford algebra $\mathcal{C}l_m$, they can all be realized in the same way for spinor-valued functions in the complex Clifford algebra $\Clm (\C)$; we discuss spinors in the next section. 

\subsection{Irreducible representations of the Spin group}
The following three representation spaces of the Spin group are frequently used as the target spaces in Clifford analysis. The spinor representation is the most commonly used spin representation in classical Clifford analysis and the other two polynomial representations are often used in higher spin theory.
\subsubsection{Spinor representation of $Spin(m)$}
Consider the complex Clifford algebra $\mathcal{C}l_m(\mathbb{C})$ with even dimension $m=2n$. Then $\C^m$ or the space of vectors is embedded in $\Clm(\C)$ as
\begin{eqnarray*}
(x_1,x_2,\cdots,x_m)\mapsto \sum^{m}_{j=1}x_je_j:\ \C ^m\hookrightarrow \mathcal{C}l_m(\C).
\end{eqnarray*}
Define the \emph{Witt basis} elements of $\C^{2n}$ as 
$$f_j:=\displaystyle\frac{e_j-ie_{j+n}}{2},\ \ f_j^{\dagger}:=-\displaystyle\frac{e_j+ie_{j+n}}{2}.$$
Let $I:=f_1f_1^{\dagger}\dots f_nf_n^{\dagger}$. The space of \emph{Dirac spinors} is defined as
$$\mathcal{S}:=\Clm(\C)I.$$ This is a representation of $Spin(m)$ under the following action
$$\rho(s)I:=sI,\ for\ s\in Spin(m).$$
Note that $\scs$ is a left ideal of $\Clm (\C)$. For more details, we refer the reader to \cite{De}. An alternative construction of spinor spaces is given in the classical paper of Atiyah, Bott and Shapiro \cite{At}.
\subsubsection{Homogeneous harmonic polynomials on $\mathcal{H}_j(\Rm,\mathbb{C})$}
It is a well-known fact that the space of complex-valued harmonic polynomials defined on several vector variables is invariant under the action of $Spin(m)$, since the Laplacian $\Delta_m$ is an $SO(m)$ invariant operator. But it is not irreducible for $Spin(m)$. It can be decomposed into the infinite sum of $j$-homogeneous harmonic polynomials, $0\leq j<\infty$. Each of these spaces is irreducible for $Spin(m)$. This brings us the most familiar representations of $Spin(m)$: spaces of $j$-homogeneous complex-valued harmonic polynomials defined on $\mathbb{R}^m$, henceforth denoted by $\Hj:=\mathcal{H}_j(\Rm,\mathbb{C})$. The following action has been shown to be an irreducible representation of $Spin(m)$ \cite{L}: 
\begin{eqnarray*}
\rho\ :\ Spin(m)\longrightarrow Aut(\Hj),\ s\longmapsto (f(x)\mapsto f(sx\tilde{s})).
\end{eqnarray*}
This can also be realized as follows
\begin{eqnarray*}
Spin(m)\xlongrightarrow{\theta}SO(m)\xlongrightarrow{\rho} Aut(\Hj);\\
a\longmapsto O_a\longmapsto \big(f(x)\mapsto f(O_ax)\big),
\end{eqnarray*}
where $\theta$ is the double covering map and $\rho$ is the standard action of $SO(m)$ on a function $f(x)\in\Hj$ with $x\in\mathbb{R}^m$. The function $\phi(z)=(z_1+iz_m)^j$ is the highest weight vector for $\Hj (\Rm,\C)$ having highest weight $(j,0,\cdots,0)$ (for more details, see \cite{G}). Accordingly, the spin representations given by $\mathcal{H}_j(\Rm,\mathbb{C})$ are said to have integer spin $j$; we can either specify an integer spin $j$ or the degree of homogeneity $j$ of harmonic polynomials.

\subsubsection{Homogeneous monogenic polynomials on $\mathcal{C}l_m$}\hspace*{\fill}
In $\mathcal{C}l_m$-valued function theory, the previously mentioned Almansi-Fischer decomposition shows that we can also decompose the space of $j$-homogeneous harmonic polynomials as follows
$$\Hj=\Mj\oplus u\Mjj.$$
If we restrict $\Mj$ to the spinor valued subspace, we have another important representation of $Spin(m)$: the space of $j$-homogeneous spinor-valued monogenic polynomials on $\Rm$, henceforth denoted by $\Mj:=\Mj(\Rm,\mathcal{S})$. More specifically, the following action has been shown to be an irreducible representation of $Spin(m)$:
\begin{eqnarray*}
\pi\ :\ Spin(m)\longrightarrow Aut(\Mj),\ s\longmapsto (f(x)\mapsto sf(sx\tilde{s})).
\end{eqnarray*}
When $m$ is odd, in terms of complex variables $z_s=x_{2s-1}+ix_{2s}$ for all $1\leq s\leq \frac{m-1}{2}$, the highest weight vector is
$\omega_j(x)=(\bar{z_1})^jI$ for $\Mj(\Rm,\mathcal{S})$ having highest weight $(j+\frac{1}{2},\frac{1}{2},\cdots,\frac{1}{2})$, where $\bar{z_1}$ is the conjugate of $z_1$, $\mathcal{S}$ is the Dirac spinor space, and $I$ is defined as in Section $2.2.1$; for details, see \cite{L}. Accordingly, the spin representations given by $\mathcal{M}_j(\Rm,\mathcal{S})$ are said to have half-integer spin $j+\frac{1}{2}$; we can either specify a half-integer spin $j+\frac{1}{2}$ or the degree of homogeneity $j$ of monogenic spinor-valued polynomials. 

\par


\section{\textbf{The higher order higher spin operator $\mathcal{D}_{1,k}$}}

\subsection{Motivation}
We have mentioned that the Laplace operator (acting on a $\C$-valued field) is related to the Dirac operator (acting on a spinor-valued field) and they are both conformally invariant operators \cite{R1}. Moreover, the $k$th-power of the Dirac operator $D_x^k$ for $k$ a positive integer, is shown also to be conformally invariant in the spinor-valued function theory \cite{R1}. However, the Dirac operator $D_x$ and the Laplace operator are no longer conformally invariant when acting on functions taking values in the higher spin spaces, in the sense explained in the next paragraph; see \cite{B1,E}, and \cite{Ding} for the Dirac operator case. The first generalization of the Dirac operator to higher spin spaces is instead the so-called Rarita-Schwinger operator \cite{B,D}, and the generalization of the Laplace operator to higher spin spaces is the so-called higher spin Laplace or Maxwell operator given in \cite{B1,E}.  

Let us look deeper into this lack of conformal invariance of the Dirac operator $D_x$ when acting on functions taking values in the higher spin spaces. Given a function $f(x,u)\in C^{\infty}(\Rm,\Mj)$ such that $D_x f(x,u) = 0$, we apply inversion $x\longmapsto \displaystyle\frac{x}{||x||^2}$ to it. There is also a reflection of $u$ in the direction $x$ given by $\displaystyle\frac{xux}{||x||^2}$; this reflection involves $x$, which changes the conformal invariance of $D_x$ such that $D_x f(x,u) = 0$ does not hold in general. This explanation also applies for the Laplace operator $\Delta_x$ in the higher spin theory. The explanation we just mentioned further implies that the $k$th-power of the Dirac operator $D_x^k$ is not conformally invariant in the higher spin theory. In this section, we will provide the generalization of $D_x^k$ when it acts on $C^{\infty}(\Rm, U)$, where $U=\Ho$ or $U=\Mo$ depending on the order. We provide nomenclature for these higher order operators in higher spin theory. We begin by examining existence and uniqueness of conformally invariant differential operators in higher spin spaces.
\subsection{Existence of conformally invariant operators}
There is a well developed literature on the existence of conformally invariant operators \cite{Fegan, Branson, J, Slovak-Soucek, VS-ref}. In \cite{J}, Slov\'{a}k demonstrated the existence of conformally invariant differential operators in higher spin spaces. Then Sou\v{c}ek considered Slov\'{a}k's results in a form more suitable for Clifford analysis. In this section, we review Sou\v{c}ek's results. For more details, we refer the reader to \cite{VS-ref}.\\
\par
Let $M=\Rm\cup\{\infty\}$ be the conformal compactification of $\Rm$, $\Gamma_m$ be the Clifford group, and $V(m)$ be the group of Ahlfors-Vahlen matrices. We know that all conformal transformations in $\Rm,\ m>2$ can be expressed in the form $\varphi (x)=(ax+b)(cx+d)^{-1}$ with $\begin{pmatrix} a & b \\ c & d
\end{pmatrix}\in V(m)$. Let $G$ denote the identity component of the group $V(m)$. The group $G$ acts transitively on $M$ and the isotropic group of the point $0\in\Rm$ is clearly the subgroup $H$ of all matrices in G with the form
$\begin{pmatrix} a & 0 \\ c & d
\end{pmatrix}$. Hence $M\cong G/H$.
\par
For a matrix $A\in H$, the element $a\in\Gamma_m$ has a nonzero norm and can be written as the product of $\displaystyle\frac{a}{||a||}\in Spin(m)$ and $||a||\in\mathbb{R}^+$. If $\lambda$ is a dominant integral weight for $Spin(m)$ with the corresponding irreducible representation $V_{\lambda}$ and $\omega\in\mathbb{C}$ is a conformal weight, we denote $\rho_{\lambda}(\omega)$  the irreducible representation of $H$ on $V_{\lambda}$ given by
\begin{eqnarray*}
\rho_{\lambda}(\omega)(h)[v]=||a||^{-2\omega}\rho_{\lambda}(\frac{a}{||a||})[v];\ v\in V_{\lambda},\ h\in H;\ h=\begin{pmatrix} a & 0 \\ c & d
\end{pmatrix}.
\end{eqnarray*}
Below we discuss differential operators acting on sections of homogeneous vector bundles over $M=G/H$. We shall consider only bundles associated to irreducible representations of the isotropic group $H$. Hence they are specified by a highest weight $\lambda$ giving an irreducible representation of $Spin(m)$ and by a conformal weight $\omega\in\mathbb{C}$. Such a bundle will be denoted by $V_{\lambda}(\omega)$. The following lemma gives the action of $G$ on $C^{\infty}(\Rm,V_{\lambda}(\omega))$.
\begin{lemma}\cite{VS-ref}\label{action}
The action of $G$ on $C^{\infty}(\Rm,V_{\lambda}(\omega))$ is given by
\begin{eqnarray*}
[g\cdot f](x)=||cx+d||^{-2\omega}\rho_{\lambda}(\frac{\widetilde{cx+d}}{||cx+d||})f((ax+b)(cx+d)^{-1}),
\end{eqnarray*}
where $g^{-1}=\begin{pmatrix} a & b \\ c & d
\end{pmatrix}\in G$. $``\cdot"$ stands for the action of $g$ on function $f$.
\end{lemma}
We now consider conformally invariant differential operators between sections of $\Gamma(M,V_{\lambda}(\omega))$ and $\Gamma(M,V_{\lambda'}(\omega'))$ of order $\omega'-\omega$, separately for the even and odd dimension cases.
\subsubsection{Even dimension $m=2n$} 
The highest weights of fundamental representations $\Lambda^i(\C^m)$ are
$$\{\lambda_i=(1,1,\cdots,1,0,\cdots,0):\ i=1,2,\cdots,n-2, \text{first}\ i\  \text{entries\ of\ $n$-tuple\ are\ $1$}\}$$ 
   and highest weights of the basic spinor representations $\mathcal{S}^{\pm}$ of $Spin(m)$ \cite{G} are $n-$tuples
$$ \sigma^{\pm}=(\displaystyle\frac{1}{2},\displaystyle\frac{1}{2},\cdots,\pm\displaystyle\frac{1}{2}).$$
Then the $(n+1)$-tuple $(B,D_i,A,C)$ specifies the irreducible representation $\rho_{\lambda}(\omega)$ for $Spin(m)$, where
 \begin{eqnarray}\label{evenlambda}
 \lambda=\displaystyle\sum_{i=1}^{n-2}(D_i-1)\lambda_i+(A-1)\sigma^{+}+(C-1)\sigma^-
\end{eqnarray}
and the conformal weight is given by
\begin{eqnarray}\label{evenomega}
\omega=n-\big[B+\sum_{i=1}^{n-2}D_i+\frac{A+C}{2}\big].
\end{eqnarray}
Let us now state Sou\v{c}ek's theorem on classification of nonstandard operators, in the even dimension case.
\begin{theorem}\cite{VS-ref}\label{VS1} 
Let $(\lambda,\omega)$ and $(\lambda',\omega')$  be computed using Equations (\ref{evenlambda}) and (\ref{evenomega}), where the positive integers $D_i, A, C, D_i',A',C', i=1,2,\cdots,n-2,$ may adopt any values in the columns to their right in the following table, and where $(\lambda',\omega')$ are determined by primed coefficients. In the table, $a,b,c,d_i,i=1,2,\cdots,n-2$ are nonnegative integers, $d=\sum_id_i$,  and the integer $e$ is defined by $e=a+b+c+d$. 

\begin{center}
\begin{tabular}{ c c c c c c c } 
 B & $-b-d$ & $-b-d+d_{n-2}$ & $\cdots$ & $-b-d_1$ & $-b$ & b \\ 
 $B'$ & $-e$ & $-e-d_{n-2}$ &$\cdots$ & $-e-d+d_1$ &$-e-d$ &$-e-d$ \\ 
 $D_1=D_1' $& b & b & $\cdots$& b & $b+d_1$ & $d_1$\\ 
 $D_2=D_2'$ & $d_1$ & $d_1$ & $\cdots$ & $d_1+d_2$ & $d_2$ & $d_2$\\
 $\vdots$ & $\vdots$ & $\vdots$ &$\vdots$ & $\vdots$ & $\vdots$ & $\vdots$ \\
 $D_{n-3}=D_{n-3}'$ & $d_{n-4}$ & $d_{n-4}$ &$\cdots$ & $d_{n-3}$ & $d_{n-3}$ & $d_{n-3}$\\
 $D_{n-2}=D_{n-2}'$ & $d_{n-3}$ & $d_{n-3}+d_{n-2}$ &$\cdots$ & $d_{n-2}$ & $d_{n-2}$ & $d_{n-2}$\\
 $A=C'$ & $a+d_{n-2}$ & a &$\cdots$ & a & a & a \\
 $C=A'$ & $c+d_{n-2}$ & c &$\cdots$ &c & c & c 
\end{tabular}
\end{center}
Then there exists (up to a multiple) unique nontrivial conformally invariant differential operators between sections of $\Gamma(M,V_{\lambda}(\omega))$ and $\Gamma(M,V_{\lambda'}(\omega'))$; its order is equal to $\omega'-\omega$. This is a complete list of the so-called nonstandard conformally invariant differential operators on spaces of even dimension.
\end{theorem}
\subsubsection{Odd dimension $m=2n+1$}
The highest weights of fundamental representations $\Lambda^i(\C^m)$ are
$$\{\lambda_i=(1,1,\cdots,1,0,\cdots,0):\ i=1,2,\cdots,n-1, \text{where\ the\ first}\ i\  \text{entries\ of\ $n$-tuple\ are\ $1$}\}$$ 
   and highest weights of the basic spinor representation $\mathcal{S}$ of $Spin(m)$ \cite{G} are $n-$tuples
$$ \sigma=(\displaystyle\frac{1}{2},\displaystyle\frac{1}{2},\cdots,\displaystyle\frac{1}{2}).$$
Then the $(n+1)$-tuple $(B,D_i,A,C)$ specifies the irreducible representation $\rho_{\lambda}(\omega)$ for $Spin(m)$, where
\begin{eqnarray}\label{oddlambda}
 \lambda=\displaystyle\sum_{i=1}^{n-1}(D_i-1)\lambda_i+(A-1)\sigma
\end{eqnarray}
and the conformal weight is given by
\begin{eqnarray}\label{oddomega}
\omega=\frac{2n+1}{2}-\big[B+\sum_{i=1}^{n-1}D_i+\frac{A}{2}\big].
\end{eqnarray}
Let us now state Sou\v{c}ek's theorem on classification of nonstandard operators, now in the odd dimension case.


\begin{theorem}\cite{VS-ref}\label{VS2} 
Let $(\lambda,\omega)$ and $(\lambda',\omega')$  be computed using Equations (\ref{oddlambda}) and (\ref{oddomega}), where the positive integers $D_i, A, C, D_i',A',C', i=1,2,\cdots,n-1,$ may adopt any values in the columns to their right in the following table, and where $(\lambda',\omega')$ are determined by primed coefficients. In the table, $a,b,c,d_i,i=1,2,\cdots,n-1$ are nonnegative half-integers or integers (at least one of them being half integral), $d=\sum_id_i$, and the integer $e$ is defined by $e=a+b+d$. 
%

\begin{center}
\begin{tabular}{ c c c c c c c } 
 B & $-b-d$ & $-b-d+d_{n-1}$ & $\cdots$ & $-b-d_1$ & $-b$ & b \\ 
 $B'$ & $-e$ & $-e-d_{n-1}$ &$\cdots$ & $-e-d+d_1$ &$-e-d$ &$-e-d$ \\ 
 $D_1=D_1' $& b & b & $\cdots$& b & $b+d_1$ & $d_1$\\ 
 $D_2=D_2'$ & $d_1$ & $d_1$ & $\cdots$ & $d_1+d_2$ & $d_2$ & $d_2$\\
 $\vdots$ & $\vdots$ & $\vdots$ &$\vdots$ & $\vdots$ & $\vdots$ & $\vdots$ \\
 $D_{n-2}=D_{n-2}'$ & $d_{n-3}$ & $d_{n-3}$ &$\cdots$ & $d_{n-2}$ & $d_{n-2}$ & $d_{n-2}$\\
 $D_{n-1}=D_{n-1}'$ & $d_{n-2}$ & $d_{n-2}+d_{n-1}$ &$\cdots$ & $d_{n-1}$ & $d_{n-1}$ & $d_{n-1}$\\
 $A=A'$ & $a+d_{n-2}$ & a &$\cdots$ & a & a & a 
 \end{tabular}
\end{center}
Then there exists (up to a multiple) unique nontrivial conformally invariant differential operators between sections of $\Gamma(M,V_{\lambda}(\omega))$ and $\Gamma(M,V_{\lambda'}(\omega'))$; its order is equal to $\omega'-\omega$. This is a complete list of the so-called nonstandard conformally invariant differential operators on spaces of odd dimension.
\end{theorem}
\subsubsection{Applications to our cases}
Theorem \ref{VS1} and \ref{VS2} show existence of conformally invariant differential operators as follows.
\begin{theorem} \cite{VS-ref}\label{VS}
Let $(\lambda,\omega)$ and $(\lambda',\omega')$ be one of couples for which there is a (nontrivial) invariant differential operator 
$$D:\ \Gamma(M,V_{\lambda}(\omega))\longrightarrow \Gamma(M,V_{\lambda'}(\omega'))$$
(the nonstandard operators are listed in Theorem \ref{VS1} and \ref{VS2}; the complete list is in \cite{J}).
\par
Let $T_{\lambda,\omega}(g),\ g^{-1}=\begin{pmatrix} a & b \\ c & d
\end{pmatrix}\in G$ (similarly for $T_{\lambda',\omega'}(g)$) be the operator acting on smooth maps from $\Rm$ to $V_{\lambda}$ by
\begin{eqnarray*}
\big[T_{\lambda,\omega}(g)f\big](x)=||cx+d||^{-2\omega}\rho_{\lambda}(\frac{\widetilde{cx+d}}{||cx+d||})\big[f((ax+b)(cx+d)^{-1})\big].
\end{eqnarray*}
Then
\begin{eqnarray*}
D(T_{\lambda,\omega}(g)f)=T_{\lambda',\omega'}(g)(Df),\ g\in G,\ f\in C^{\infty}(\Rm,V_{\lambda}).
\end{eqnarray*}
\end{theorem}

This paper considers differential operators acting on functions $f(x,u)\in C^{\infty}(\Rm,\Hj)$ or $f(x,u)\in C^{\infty}(\Rm,\Mj)$. Here we only show existence of conformally invariant differential operators on spaces of \emph{even dimension $m$}; the odd dimensional case is similar. We work out the allowable highest weights, conformal weights, and orders on operators acting on these function spaces.
\par
From Theorem \ref{VS1}, we notice that highest weight $\lambda$ is determined by $D_i,\ A$ and $C$. From the table, we also have $\lambda=\lambda'$. In other words, conformally invariant operators only exist between $C^{\infty}(\Rm,\Hj)$ and itself or $C^{\infty}(\Rm,\Mj)$ and itself. We consider each in turn.\\
\par
\textbf{Integer Spin Case: $C^{\infty}(\Rm,\Hj)\longrightarrow C^{\infty}(\Rm,\Hj)$}\\
\par
As an irreducible representation of $Spin(m)$, $\Hj$ has highest weight of $n$-tuple $\lambda=\lambda'=(j,0,\cdots,0).$ The group action $\rho_{\lambda}$ is defined in Section $2.2.2$. From the table in Theorem \ref{VS1} and
\begin{eqnarray*}
&&\lambda=\displaystyle\sum_{i=1}^{n-2}(D_i-1)\lambda_i+(A-1)\sigma^{+}+(C-1)\sigma^-,\\
&&\lambda_i=(1,1,\cdots,1,0,\cdots,0),\ i=1,2,\cdots,n-2,\\
&&\sigma^{\pm}=(\displaystyle\frac{1}{2},\displaystyle\frac{1}{2},\cdots,\pm\displaystyle\frac{1}{2}),
\end{eqnarray*}
we know that $D_1=j+1,\ D_i=1,\ i=2,\cdots,n-2$ and $A=C=1$. There is exactly one possibility for all entries but the last column in the table, for which there is a sequence of possibilities indexed by a nonnegative integer $b$. The last column corresponds to the $(2b+2n+2j-2)$-th order conformally invariant differential operator, with 
$$d=j+n-2,\ a=c=1,\ e=b+j+n,\ B=b,\ B'=-b-2j-2n+2,$$
and
conformal weights $\omega=1-b-j$ and $\omega'=b+j+2n-1$. Hence, we have
\begin{eqnarray*}
&&\mathcal{D}_{1,2b+2n+2j-2} T_{\lambda,1-b-j}=T_{\lambda,b+j+2n-1}\mathcal{D}_{1,2b+2n+2j-2},\\
&&i.e.,\ \mathcal{D}_{1,2b+2n+2j-2}||cx+d||^{2j+2b-2}=||cx+d||^{-2b-2j-4n+2}\mathcal{D}_{1,2b+2n+2j-2}.
\end{eqnarray*}
To make the above intertwining operators coincide with the forms of the intertwining operators we have at the end of Section 3, we let $2b+2n+2j-2=2s$ and since $m=2n$, we have
$$\mathcal{D}_{1,2s}||cx+d||^{2s-m}=||cx+d||^{-m-2s}\mathcal{D}_{1,2s}.$$\\
 \par
\textbf{Half-integer Spin Case: $C^{\infty}(\Rm,\Mj)\longrightarrow C^{\infty}(\Rm,\Mj)$}\\
\par
 As an irreducible representation of $Spin(m)$, $\Mj$ has highest weight as $n$-tuple $\lambda=\lambda'=(j+\displaystyle\frac{1}{2},\displaystyle\frac{1}{2},\cdots,\pm\displaystyle\frac{1}{2}).$ The group action $\rho_{\lambda}$ is the action defined as in Section $2.2.3$. From the table and
\begin{eqnarray*}
&&\lambda=\displaystyle\sum_{i=1}^{n-2}(D_i-1)\lambda_i+(A-1)\sigma^{+}+(C-1)\sigma^-,\\
&&\lambda_i=(1,1,\cdots,1,0,\cdots,0),\ i=1,2,\cdots,n-2,\\
&&\sigma^{\pm}=(\displaystyle\frac{1}{2},\displaystyle\frac{1}{2},\cdots,\pm\displaystyle\frac{1}{2}),
\end{eqnarray*}
we can find that $D_1=j+1, D_i=1, i=2,\cdots,n-2$ and $A=1,C=0$( or $A=0, C=1$ depending on the last entry of $\lambda$ is $\displaystyle\frac{1}{2}$ or $-\displaystyle\frac{1}{2}$). Similar as the previous case, there is just one possibility for all but the last column in the table and there is a sequence of possibilities indexed by a nonnegative integer $b$ for the last column. The last column corresponds to the $(2j+2n+2b-3)$-th order conformally invariant differential operator, with
$$d=j+n-2,\ a=1,c=0\ (or\ a=0,\ c=1),\ e=b+j+n-1,\ B=b,\ B'=-2j-2n-b+3,$$
and conformal weights $\omega=-b-j+\displaystyle\frac{3}{2}$ and $\omega'=j+2n+b-\displaystyle\frac{3}{2}$. Hence, we have
\begin{eqnarray*}
\mathcal{D}_{1,2j+2n+2b-3} T_{\lambda,-b-j+\frac{3}{2}}=T_{\lambda,j+2n+b-\frac{3}{2}}\mathcal{D}_{1,2j+2n+2b-3},
\end{eqnarray*}
in other words,
\begin{eqnarray*}
\mathcal{D}_{1,2j+2n+2b-3}||cx+d||^{2b+2j-3}\frac{\widetilde{cx+d}}{||cx+d||}=||cx+d||^{-2j-4n-2b+3}\frac{\widetilde{cx+d}}{||cx+d||}\mathcal{D}_{1,2j+2n+2b-3}.
\end{eqnarray*}
To make the above intertwining operators coincide with the intertwining operators we have at the end of Section 3, we let $2j+2n+2b-3=2s+1$ and since $m=2n$, we have
$$\mathcal{D}_{1,2s+1}\displaystyle\frac{\widetilde{cx+d}}{||cx+d||^{m-2s}}=\displaystyle\frac{\widetilde{cx+d}}{||cx+d||^{m+2s+2}}\mathcal{D}_{1,2s+1}.$$

Similar arguments apply for the odd dimensional cases. This establishes existence of the conformally invariant differential operators we wish to consider. Further, even order conformally invariant differential operators only exist between $C^{\infty}(\Rm,\Hj)$ and odd order ones only exist between $C^{\infty}(\Rm,\Mj)$. Intertwining operators of conformally invariant differential operators in Theorem \ref{interwining op of D1k} can also be recovered. Once we establish conformal invariance of the operators that we construct between the desired higher spin spaces, uniqueness up to multiplicative constant of these higher order higher spin operators is established by the preceding theorems.
\subsection{Construction and conformal invariance}

We have established by arguments of Slov\'{a}k \cite{J} and Sou\v{c}ek \cite{VS-ref}, for integers $j\ge0$ and $k>0$ there exist conformally invariant differential operators in the higher spin setting $$\mathcal{D}_{j,k}:\ C^{\infty}(\Rm, U)\longrightarrow C^{\infty}(\Rm, U),$$ where $U=\mathcal{H}_j$ if $k$ is even and $U=\mathcal{M}_j$ if $k$ is odd. We introduce some nomenclature suggestive of massless spin fields in mathematical physics, which we hope is adopted by others studying higher spin theory in Clifford analysis. As a Spin representation $\mathcal{H}_j$ is associated with integer spin $j$ and particles of integer spin are called bosons, so the operators $\mathcal{D}_{j,k}:\ C^{\infty}(\Rm, \mathcal{H}_j)\longrightarrow C^{\infty}(\Rm, \mathcal{H}_j)$ are named \emph{bosonic operators}. Thus in the spin 0 case we have the Laplace operator and its $k$-powers, the spin 1 case the Maxwell operator and its generalization to order $k=2n$, and general higher spin Laplace operators and their generalization to order $k=2n$.  Correspondingly, as a Spin representation $\mathcal{M}_j$ is associated with half-integer spin $j+\frac{1}{2}$ and particles of half-integer spin are called fermions, so the operators $\mathcal{D}_{j,k}:\ C^{\infty}(\Rm, \mathcal{M}_j)\longrightarrow C^{\infty}(\Rm, \mathcal{M}_j)$ are named \emph{fermionic operators}. Thus in the spin $\frac{1}{2}$ case we have the Dirac operator and its $k=2n+1$ powers, the spin $\frac{3}{2}$ case the simplest Rarita-Schwinger operator and its generalization to order $k=2n+1$, and general Rarita-Schwinger operators and their generalization to order $k=2n+1$. Note that our notation indexes according to degree of homogeneity of the target space $j$ and differential order $k$, so fractions are not used in the notation; if we indexed according to spin, fractional spins would need to be used for odd order operators.


We will consider the higher order spin 1 and spin $\frac{3}{2}$ operators $\DD: C^{\infty}(\Rm, U)\longrightarrow C^{\infty}(\Rm, U),$
where $U=\Ho$ for $k$ even and $U=\mathcal{M}_1$ for $k$ odd. Note that the target space $U$ here is a function space. That means any element in $C^{\infty}(\Rm,U)$ is of the form $f(x,u)$ with $f(x,u)\in U$ for each fixed $x\in\Rm$ and $x$ is the variable which $\DD$ acts on. The construction and conformal invariance of these two operators are considered as follows.
 
\subsubsection*{$k$ even, $k=2n$, $n>1$ (The bosonic case)}

\begin{theorem}
For positive integer $n$, the unique $2n$-th order conformally invariant differential operator of spin-$1$ $\mathcal{D}_{1,2n}:C^{\infty}(\mathbb{R}^m,\mathcal{H}_1)\longrightarrow C^{\infty}(\mathbb{R}^m,\mathcal{H}_1)$ has the following form, up to a multiplicative constant:
\begin{eqnarray*}
\mathcal{D}_{1,2n}=\Delta_x^n-\frac{4n}{m+2n-2}\langle u,D_x\rangle \langle D_u,D_x\rangle\Delta_x^{n-1}.
\end{eqnarray*}
\end{theorem}
For the case $n=1$, we retrieve the Maxwell operator from \cite{E}.\\
\par
Our proof of conformal invariance of this operator follows closely the method of \cite{E}. In order to explain what conformal invariance means, we begin with the concept of a generalized symmetry (see for instance \cite{Eastwood}):
\begin{definition}
An operator $\eta_1$ is a generalized symmetry for a differential operator $\mathcal{D}$ if and only if there exists another operator $\eta_2$ such that $\mathcal{D}\eta_1=\eta_2\mathcal{D}$. Note that for $\eta_1=\eta_2$, this reduces to a definition of a (proper) symmetry: $\mathcal{D}\eta_1=\eta_1\mathcal{D}$.
\end{definition}
One determines the first order generalized symmetries of an operator, which span a Lie algebra \cite{E,Miller}. In this case, the first order symmetries will span a Lie algebra isomorphic to the conformal Lie algebra $\mathfrak{so}(1,m+1)$; in this sense, the operators we consider are conformally invariant. The operator $\mathcal{D}_{1,2n}$ is $\mathfrak{so}(m)$-invariant (rotation-invariant) because it is the composition of $\mathfrak{so}(m)$-invariant (rotation-invariant) operators, which means the angular momentum operators $L_{ij}^x+L_{i,j}^u$ that generate these rotations are proper symmetries of $\mathcal{D}_{1,2n}$.
 The infinitesimal translations $\partial_{x_j}, j=1,\cdots , n,$ corresponding to linear momentum operators are proper  symmetries of $\mathcal{D}_{1,2n}$; this is an alternative way to say that $\mathcal{D}_{1,2n}$ is invariant under translations that are generated by these infinitesimal translations. Readers familiar with quantum mechanics will recognize the connection to isotropy and homogeneity of space, the rotational and translational invariance of Hamiltonian, and the conservation of angular and linear momentum \cite{Sakurai}; see also \cite{Br} concerning Rarita-Schwinger operators. 
 
 The remaining two of the first order generalized symmetries of $\mathcal{D}_{1,2n}$ are the Euler operator and special conformal transformations. The Euler operator $\mathbb{E}_x$ that measures degree of homogeneity in $x$ is a generalized symmetry because $\mathcal{D}_{1,2n}\mathbb{E}_x=(\mathbb{E}_x+2n)\mathcal{D}_{1,2n}$; this is an alternative way to say that $\mathcal{D}_{1,2n}$ is invariant under dilations, which are generated by the Euler operator. The special conformal transformations are defined in Lemma \ref{SCT} in terms of harmonic inversion for $\Ho$-valued functions; harmonic inversion is defined in Definition \ref{HI} and is an involution mapping solutions of $\mathcal{D}_{1,2n}$ to $\mathcal{D}_{1,2n}$. Readers familiar with conformal field theory will recognize that invariance under dilation corresponds to scale-invariance and that special conformal transformations are another class of conformal transformations arising on spacetime \cite{CFT}. An alternative method of proving conformal invariance of $\mathcal{D}_{1,2n}$ is to prove the invariance of $\mathcal{D}_{1,2n}$ under those finite transformations generated by these first order generalized symmetries (rotations, dilations, translations, and special conformal transformations) to show invariance of $\mathcal{D}_{1,2n}$ under actions of the conformal group; this may be phrased in terms of M\"obius transformations and the Iwasawa decomposition. However, the first-order generalized symmetry method emphasizes the connection to mathematical physics and is more amenable to our proof of a certain property of harmonic inversion. It is also that used by earlier authors \cite{B1,E}. 

\begin{definition} \label{HI}
The harmonic inversion is a conformal transformation defined as
\begin{eqnarray*}
\mathcal{J}_{2n}:C^{\infty}(\mathbb{R}^m,\mathcal{H}_1)\longrightarrow C^{\infty}(\mathbb{R}^m,\mathcal{H}_1):f(x,u)\mapsto \mathcal{J}_{2n}[f](x,u):=||x||^{2n-m}f(\frac{x}{||x||^2},\frac{xux}{||x||^2}).
\end{eqnarray*}
\end{definition}
Note that this inversion consists of Kelvin inversion $\mathcal{J}$ on $\mathbb{R}^m$ in the variable $x$ composed with a reflection $u\mapsto \omega u\omega$ acting on the dummy variable $u$ (where $x=||x||\omega$) and a multiplication by a conformal weight term $||x||^{2n-m}$; it satisfies $\mathcal{J}_{2n}^2=1.$\\
\par
Then we have the special conformal transformation defined in the following lemma. The definition is an infinitesimal version of the fact that finite special conformal transformations consist of a translation preceded and followed by an inversion \cite{CFT}: an infinitesimal translation preceded and followed by harmonic inversion. The second equality in the lemma shares some terms in common with the generators of special conformal transformations in conformal field theory \cite{CFT}, and is a particular case of a result in \cite{ER}.

\begin{lemma} \label{SCT}
The special conformal transformation defined as $\mathcal{C}_{2n}:=\mathcal{J}_{2n}\partial_{x_j}\mathcal{J}_{2n}$ satisfies
\begin{eqnarray*}
\mathcal{C}_{2n}=2\langle u,x\rangle\partial_{u_j}-2u_j\langle x,D_u\rangle +||x||^2\partial_{x_j}-x_j(2\mathbb{E}_x+m-2n).
\end{eqnarray*}
\end{lemma}
\begin{proof}
A similar calculation as in \emph{Proposition A.1} in \cite{B1} will show the conclusion.
\end{proof}
Then, we have the main proposition as follows.
\begin{proposition}\label{propeven}
The special conformal transformations $\mathcal{C}_{2n}$, with $j\in\{1,2,\dots,m\}$ are generalized symmetries of $\mathcal{D}_{1,2n}$. More specifically,
\begin{eqnarray*}
[\mathcal{D}_{1,2n},\mathcal{C}_{2n}]=-4nx_j\mathcal{D}_{1,2n}.
\end{eqnarray*}
 In particular, this shows that 
\begin{eqnarray}\label{crucialeven}
\mathcal{J}_{2n}\mathcal{D}_{1,2n}\mathcal{J}_{2n}=||x||^{4n}\mathcal{D}_{1,2n},
\end{eqnarray}
 which is the generalization of the case of the classical higher order Laplace operator $\Delta_x^n$ \cite{Ax}. This also implies $\mathcal{D}_{1,2n}$ is invariant under inversion.
 \end{proposition}
 If the main proposition holds, then the conformal invariance can be summarized in the following theorem:
 \begin{theorem}\label{theoremeven}
 The first order generalized symmetries of $\mathcal{D}_{1,2n}$ are given by:
 \begin{enumerate}
 \item The infinitesimal rotation $L_{i,j}^x+L_{i,j}^u$, with $1\leq i<j\leq m$.
 \item The shifted Euler operator $(\mathbb{E}_x+\displaystyle\frac{m-2n}{2})$.
 \item The infinitesimal translations $\partial_{x_j}$, with $1\leq j\leq m$.
 \item The special conformal transformations $\mathcal{J}_{2n}\partial_{x_j}\mathcal{J}_{2n}$, with $1\leq j\leq m$.
 \end{enumerate}
 These operators span a Lie algebra which is isomorphic to the conformal Lie algebra $\mathfrak{so}(1,m+1)$, whereby the Lie bracket is the ordinary commutator.
 \end{theorem}
\begin{proof}
The proof is similar as in \cite{ER} via transvector algebras. Notice that the shift in the shifted Euler operator $\mathbb{E}_x+\omega$ defines the conformal weight (defined in Section $3.2$) $\omega=\displaystyle\frac{m-2n}{2}$.
\end{proof}
\subsubsection*{Detailed proof of Proposition \ref{propeven}:}
First, let us prove a few technical lemmas. It is worth pointing out that since we are dealing with degree-$1$ homogeneous polynomials in $u$, terms involving second derivatives with respect to $u$ disappear.
\begin{lemma}\label{lemma1}
For all $1\leq j\leq m$, we have
\begin{eqnarray*}
[\Delta_x^n,\mathcal{C}_{2n}]=-4nx_j\Delta_x^n+4n\langle u,D_x\rangle\partial_{u_j}\Delta_x^{n-1}-4nu_j\langle D_u,D_x\rangle\Delta_x^{n-1}.
\end{eqnarray*}
\end{lemma}
\begin{proof}
We prove this by induction. First, we have (\cite{B1})
\begin{eqnarray*}
[\Delta_x,\mathcal{C}_{2}]=-4x_j\Delta_x+4\langle u,D_x\rangle\partial_{u_j}-4u_j\langle D_u,D_x\rangle.
\end{eqnarray*}
Assuming the lemma is true for $\Delta^{n-1}$, applying the fact that for general operators $A,\ B$ and $C$
\begin{eqnarray*}
[AB,C]&=&A[B,C]+[A,C]B
\end{eqnarray*}
and
\begin{eqnarray*}
\mathcal{C}_{2n}&=&\mathcal{C}_{2}+(2n-2)x_j,
\end{eqnarray*}
we have 
\begin{eqnarray*}
&&[\Delta_x^n,\mathcal{C}_{2n}]
=\Delta_x^{n-1}[\Delta_x,\mathcal{C}_{2n}]+[\Delta_x^{n-1},\mathcal{C}_{2n}]\Delta_x.
\end{eqnarray*}
Since
$$\mathcal{C}_{2n}=\mathcal{C}_{2n-2}+2x_j,$$
a straightforward calculation leads to the conclusion.
\end{proof}

\begin{lemma}\label{lemma2}
For all $1\leq j\leq m,$ we have
\begin{eqnarray*}
&&[\langle u,D_x\rangle\langle D_u,D_x\rangle \Delta_x^{n-1},\mathcal{C}_{2n}] \\
&=&-4nx_j\langle u,D_x\rangle\langle D_u,D_x\rangle \Delta_x^{n-1}+(m+2n-2)\big(\langle u,D_x\rangle\partial_{u_j}-u_j\langle D_u,D_x\rangle\big)\Delta_x^{n-1}.
\end{eqnarray*}
\end{lemma}
\begin{proof}
First, we have \cite{B1}:
\begin{eqnarray*}
&&[\langle u,D_x\rangle\langle D_u,D_x\rangle, \mathcal{C}_{2}]\\
&=&2||u||^2\partial_{u_j}\langle D_u,D_x\rangle-4x_j\langle u,D_x\rangle\langle D_u,D_x\rangle+\big(\langle u,D_x\rangle\partial_{u_j}-u_j\langle D_u,D_x\rangle\big)(2\mathbb{E}_u+m-2)\\
&=&-4x_j\langle u,D_x\rangle\langle D_u,D_x\rangle+m\big(\langle u,D_x\rangle\partial_{u_j}-u_j\langle D_u,D_x\rangle\big),
\end{eqnarray*}
 Then
\begin{eqnarray*}
&&[\langle u,D_x\rangle\langle D_u,D_x\rangle \Delta_x^{n-1},\mathcal{C}_{2n}] \\
&=&\langle u,D_x\rangle\langle D_u,D_x\rangle[\Delta_x^{n-1},\mathcal{C}_{2n}]+[\langle u,D_x\rangle\langle D_u,D_x\rangle,\mathcal{C}_{2n}]\Delta_x^{n-1}
\end{eqnarray*}
together with the previous lemma proves the conclusion.
\end{proof}

With the help of \emph{Lemma} \ref{lemma1} and \ref{lemma2}, a straightforward calculation shows that
$$[\mathcal{D}_{1,2n},\mathcal{C}_{2n}]=-4nx_j\mathcal{D}_{1,2n}.$$
Since $\mathcal{D}_{1,2n}$ is conformally invariant by Theorem \ref{theoremeven} and Slovak's results provide the uniqueness and intertwining operators of conformally invariant differential operators, we have 
$$\mathcal{J}_{2n}\mathcal{D}_{1,2n}\mathcal{J}_{2n}=||x||^{4n}\mathcal{D}_{1,2n}$$
from the intertwining operators under (harmonic) inversion. This is a generalization of the Laplacian case \cite{Ax}.

\subsubsection*{$k$ odd, $k=2n-1$, $n>1$ (The fermionic case)}

\begin{theorem}
For positive integer $n$, the unique $(2n-1)$-th order conformally invariant differential operator of spin-$\frac{3}{2}$ $\mathcal{D}_{1,2n-1}:C^{\infty}(\mathbb{R}^m,\mathcal{M}_1)\longrightarrow C^{\infty}(\mathbb{R}^m,\mathcal{M}_1)$ has the following form, up to a multiplicative constant:
\begin{eqnarray*}
\mathcal{D}_{1,2n-1}=D_x\Delta_x^{n-1}-\frac{2}{m+2n-2}u\langle D_u,D_x\rangle \Delta_x^{n-1}-\frac{4n-4}{m+2n-2}\langle u,D_x\rangle \langle D_u,D_x\rangle \Delta_x^{n-2}D_x.
\end{eqnarray*}
\end{theorem}
When $n=1$, we have the Rarita-Schwinger operator appearing in \cite{B, D} and elsewhere.\\
\par
The same strategy in the even case applies: we only must show the special conformal transformation defined below is a generalized symmetry of $\mathcal{D}_{1,2n-1}$. We have the definition for monogenic inversion as follows.
\begin{definition}
Monogenic inversion is a conformal transformation defined as
\begin{eqnarray*}
&&\mathcal{J}_{2n+1}\ :\ C^{\infty}(\mathbb{R}^m,\mathcal{M}_1)\longrightarrow C^{\infty}(\mathbb{R}^m,\mathcal{M}_1);\\
&&f(x,u)\mapsto \mathcal{J}_{2n+1}[f](x,u):=\frac{x}{||x||^{m-2n}}f(\frac{x}{||x||^2},\frac{xux}{||x||^2}).
\end{eqnarray*}
\end{definition}
Note that this inversion also consists of Kelvin inversion $\mathcal{J}$ on $\mathbb{R}^m$ in the variable $x$ composed with a reflection $u\mapsto \omega u\omega$ acting on the dummy variable $u$ (where $x=||x||\omega$) and a multiplication of a conformal weight term $\displaystyle\frac{x}{||x||^{m-2n}}$; it satisfies $\mathcal{J}_{2n+1}^2=-1$ instead. Similarly, monogenic inversion is an involution mapping solutions for $\Do$ to solutions for $\Do$ (\cite{Ro}). Then we have the following lemma.
\begin{lemma}
The special conformal transformation is defined as
\begin{eqnarray*}
\mathcal{C}_{2n-1}:=\mathcal{J}_{2n-1}\partial_{x_j}\mathcal{J}_{2n-1}=-e_jx-2\langle u,x\rangle\partial_{u_j}+2u_j\langle x,D_u\rangle -||x||^2\partial_{x_j}+x_j(2\mathbb{E}_x+m-2n),
\end{eqnarray*}
$\mathcal{C}_{2n-1}=\mathcal{C}_{2n-3}-2x_j=-\mathcal{C}_{2n-2}-e_jx-2x_j$.
\end{lemma}
\begin{proof}
As similar calculation as in \emph{Proposition A.1} in \cite{B1} will show the conclusion.
\end{proof}
Then we arrive at the main proposition, stating that the special conformal transformations are generalized symmetries of operator $\mathcal{D}_{1,2n-1}$.

\begin{proposition}\label{propodd}
The special conformal transformations $\mathcal{C}_{2n-1}$, with $j\in\{1,2,\dots,m\}$ are generalized symmetries of $\mathcal{D}_{1,2n-1}$. More specifically,
\begin{eqnarray*}
[\mathcal{D}_{1,2n-1},\mathcal{C}_{2n-1}]=(4n-2)x_j\mathcal{D}_{1,2n-1}.
\end{eqnarray*}
In particular, this shows that $\mathcal{J}_{2n-1}\mathcal{D}_{1,2n-1}\mathcal{J}_{2n-1}=||x||^{4n-2}\mathcal{D}_{1,2n-1}$, which is the generalization of the case of the classical higher order Dirac operator $D_x^{2n-1}$ \cite{Ax}. This also implies $\mathcal{D}_{1,2n-1}$ is invariant under inversion.
\end{proposition}
 \begin{theorem}\label{theoremodd}
 The first order generalized symmetries of $\mathcal{D}_{1,2n-1}$ are given by:
 \begin{enumerate}
 \item The infinitesimal rotation $L_{i,j}^x+L_{i,j}^u$, with $1\leq i<j\leq m$.
 \item The shifted Euler operator $(\mathbb{E}_x+\displaystyle\frac{m-2n+1}{2})$.
 \item The infinitesimal translations $\partial_{x_j}$, with $1\leq j\leq m$.
 \item The special conformal transformations $\mathcal{J}_{2n-1}\partial_{x_j}\mathcal{J}_{2n-1}$, with $1\leq j\leq m$.
 \end{enumerate}
 These operators span a Lie algebra which is isomorphic to the conformal Lie algebra $\mathfrak{so}(1,m+1)$, whereby the Lie bracket is the ordinary commutator.
 \end{theorem}
 \subsubsection*{Detailed proof of Proposition \ref{propodd}:}
To prove Proposition \ref{propodd}, as in the even case, we need a few technical lemmas.

\begin{lemma}\label{lemma 5}
For all $1 \leq j \leq m$, we have
\begin{eqnarray*}
&&[D_x\Delta_x^{n-1},\mathcal{C}_{2n-1}] \\
&=&(4n-2)x_jD_x\Delta_x^{n-1}+(4n-4)\big(u_j\langle D_u,D_x\rangle-\langle u, D_x\rangle \partial_{u_j}\big)D_x\Delta_x^{n-2}-2u\partial_{u_j}\Delta_x^{n-1}.
\end{eqnarray*}
\end{lemma}

\begin{lemma}\label{lemma 6}
For all $1 \leq j \leq m$, we have
\begin{eqnarray*}
&&[u\langle D_u,D_x\rangle \Delta_x^{n-1}, \mathcal{C}_{2n-1}] \\
&=&(4n-2)x_ju\langle D_u,D_x\rangle \Delta_x^{n-1}-(m+2n-2)u\partial_{u_j}\Delta_x^{n-1}-(2n-2)ue_j\langle D_u,D_x \rangle \Delta_x^{n-2}.
\end{eqnarray*}
\end{lemma}

\begin{lemma}\label{lemma 7}
For all $1 \leq j \leq m$, we have
\begin{eqnarray*}
&&[\langle u, D_x\rangle \langle D_u,D_x\rangle \Delta_x^{n-2}D_x,\mathcal{C}_{2n-1}]
=(4n-2)x_j\langle u, D_x\rangle \langle D_u,D_x\rangle \Delta_x^{n-2}D_x\\
&&-(m+2n-2)\big( \langle u,D_x\rangle \partial_{u_j}-u_j\langle D_u,D_x\rangle \big)\Delta_x^{n-2}D_x+ue_j\langle D_u,D_x\rangle \Delta_x^{n-2}D_x.
\end{eqnarray*}
\end{lemma}

We combine \emph{Lemma} \ref{lemma 5}, \ref{lemma 6} and \ref{lemma 7} to get 
\begin{eqnarray*}
[\mathcal{D}_{1,2n-1},\mathcal{C}_{2n-1}]=(4n-2)x_j\mathcal{D}_{1,2n-1}.
\end{eqnarray*}
This implies $\mathcal{J}_{2n-1}\mathcal{D}_{1,2n-1}\mathcal{J}_{2n-1}=||x||^{4n-2}\mathcal{D}_{1,2n-1}$. 

\subsubsection*{Conformal Invariance and Intertwining Operators, Both Cases}
Strictly speaking, Theorem \ref{VS} together with the constructions in this subsection can provide the intertwining operators for the bosonic and fermionic operators in this paper. However, for the sake of concreteness and to highlight the alternative approach centering upon M\"obius transformations, here we rely on the Iwasawa decomposition for M\"obius transformations to determine these intertwining operators. Let $\mathcal{D}_{1,k,x,u}$ and $\mathcal{D}_{1,k,y,w}$ be the higher order higher spin operators with respect to $x,\ u$ and $y,\ w$, respectively and $y=\phi(x)=(ax+b)(cx+d)^{-1}$ is a M\"{o}bius transformation. Let 
\begin{eqnarray*}
J_k=\frac{\widetilde{cx+d}}{||cx+d||^{m-2n}},\ \ for\ k=2n+1;\\
J_k=\frac{1}{||cx+d||^{m-2n}},\ \ for\ k=2n;\\
J_{-k}=\frac{\widetilde{cx+d}}{||cx+d||^{m+2n+2}},\ \ for\ k=2n+1;\\
J_{-k}=\frac{1}{||cx+d||^{m+2n}},\ \ for\ k=2n,\\
\end{eqnarray*}
with $n=1,2,3,\cdots.$ See \cite{P}. Then we make the following claim.
\begin{theorem}\label{interwining op of D1k}
Let $y=\phi (x)=(ax+b)(cx+d)^{-1}$ be a M\"{o}bius transformation. Then
\begin{eqnarray*}
J_{-k}\mathcal{D}_{1,k,y,w} f(y,w)=\mathcal{D}_{1,k,x,u}J_k f(\phi(x),\frac{(cx+d)u(\widetilde{cx+d})}{||cx+d||^2}),
\end{eqnarray*}
where $w=\displaystyle\frac{(cx+d)u(\widetilde{cx+d})}{||cx+d||^2}$.
\end{theorem}
We only prove the bosonic (order $k=2n$) case, as the fermionic (order $k=2n+1$) case is similar. According to the Iwasawa decomposition, we need only prove this with respect to orthogonal transformation and inversion, since translation and dilation are trivial. Note that our argument here requires the invariance under harmonic inversion established earlier.
\subsubsection*{\textbf{Orthogonal transformations} $a\in Pin(m)$}
\begin{lemma}
If $x=ay\tilde{a},$ $u=aw\tilde{a}$, then $\mathcal{D}_{1,2n,x,u}f(x,u)=a\mathcal{D}_{1,2n,y,w}\tilde{a}f(y,w).$
\end{lemma}
\begin{proof}
\begin{eqnarray*}
&&\mathcal{D}_{1,2n,x,u}f(x,u)=\bigg(\triangle_x-\frac{4n}{m+2n-2}\langle u,D_x\rangle \langle D_u,D_x\rangle \bigg)\Delta_x^{n-1}f(x,u) \\
&=&\bigg(a\triangle_y\tilde{a}-\frac{4n}{m+2n-2}a\langle w,D_y\rangle\tilde{a}a \langle D_w,D_y\rangle\tilde{a} \bigg)a\Delta_y^{n-1}\tilde{a}f(y,w) \\
&=&a\bigg(\triangle_y-\frac{4n}{m+2n-2}\langle w,D_y\rangle \langle D_w,D_y\rangle \bigg)\Delta_y^{n-1}\tilde{a}f(y,w) \\
&=&a\mathcal{D}_{1,2n,y,w}\tilde{a}f(y,w).
\end{eqnarray*}
\end{proof}

\subsubsection*{Inversions}
\begin{lemma}\label{inversion}
Let $x=y^{-1}$ and $u=\displaystyle\frac{ywy}{||y||^2}$, then
\begin{eqnarray*}
\mathcal{D}_{1,2n,y,w}||x||^{m-2n}f(y,w)=||x||^{m+2n}\mathcal{D}_{1,2n,x,u}f(x,u).
\end{eqnarray*}
\end{lemma}
\begin{proof}
Recall that after we showed $[\mathcal{D}_{1,2n},\mathcal{J}_{2n}\partial_{x_j}\mathcal{J}_{2n}]=-4nx_j\mathcal{D}_{1,2n}$ for $\mathcal{J}_{2n}$ the harmonic inversion, we claimed and later showed that $\mathcal{J}_{2n}\mathcal{D}_{1,2n}\mathcal{J}_{2n}=||x||^{4n}\mathcal{D}_{1,2n}$. This can also be written as
\begin{eqnarray*}
\mathcal{D}_{1,2n,y,w}||x||^{m-2n}f(y,w)=||x||^{m+2n}\mathcal{D}_{1,2n,x,u}f(x,u).
\end{eqnarray*}
\end{proof}
Theorem \ref{interwining op of D1k} now follows using the Iwasawa decomposition. See \cite{Ding} for the first order case. 

\section{\textbf{Fundamental solutions of $\mathcal{D}_{1,k}$}}

To get the fundamental solutions of $\mathcal{D}_{1,k}$, we use techniques from \cite{B}. It is worth pointing out that the reproducing kernels of $\Mo$ and $\Ho$ below have simple expression, but we insist on using techniques used in \cite{B}, since it also works for more general cases when we have $\mathcal{M}_j$ or $\mathcal{H}_j$ instead. This will be found in an upcoming paper. This method only provides us the fundamental solutions up to a multiplicative constant. We prefer this constant to be determined in the more general case in an upcoming paper.
\subsubsection*{$k$ even,\ $k=2n$  (The bosonic case)} 
Recall that the reproducing kernel for $j$-homogeneous harmonic spherical polynomials $Z_{j}(u,v)$ is called the \emph{zonal spherical harmonic} of degree $j$, and is invariant under reflections (and consequently rotations) in the variables $u$ and $v$ \cite{Ax}. In our circumstance, $$Z_1(u,v)=\displaystyle\frac{(m-2)^2\omega_{m-1}}{m}\langle u,v\rangle$$ is the zonal spherical harmonic of degree $1$, where $\omega_{m-1}$ is the surface area of the $(m-1)$-dimensional unit sphere and $\langle u,v\rangle$ is the standard inner product in Euclidean space. It can be considered as the identity of $End(\mathcal{H}_1)$ and satisfies
\begin{eqnarray*}
P_1(v)=(Z_1(u,v),P_1(u))_u:=\int_{S^{m-1}} \overline{Z_1(u,v)}P_1(u)dS(u),
\end{eqnarray*}
where $(\ ,\ )_u$ denotes the Fischer inner product with respect to $u$; we define the Fischer inner product of two functions by the integral of their product over the sphere, consistent with other work in higher spin theory \cite{B,D}. A homogeneous $End(\mathcal{H}_1)$-valued $C^{\infty}$-function $x\rightarrow E(x)$ on $\mathbb{R}^m\backslash \{0\}$ satisfying $\mathcal{D}_{1,2n}E(x)=\delta(x)Z_1(u,v)$ is referred to as a fundamental solution for the operator $\mathcal{D}_{1,2n}$. We will show that such a fundamental solution has the form $E_{1,2n}(x,u,v)=c_1||x||^{2n-m}Z_1(\displaystyle\frac{xux}{||x||^2},v)$. Since $Z_1(u,v)$ is a trivial solution of $\De$, according to the invariance of $\mathcal{D}_{1,2n}$ under inversion, we obtain a non-trivial solution $\mathcal{D}_{1,2n}E_{1,2n}(x,u,v)=0$ in $\mathbb{R}^m\backslash \{0\}$. Clearly the function $E_{1,2n}(x,u,v)$ is homogeneous of degree $2n-m$ in $x$, so $\mathcal{D}_{1,2n}E_{1,2n}(x,u,v)$ is homogeneous of degree $-m$ in $x$ and it belongs to $L_1^{loc}(\mathbb{R}^m)$. Because $\delta(x)$ is the only (up to a multiple) distribution homogeneous of degree $-m$ with support at the origin, we have in the sense of distributions:
\begin{eqnarray*}
\mathcal{D}_{1,2n}E_{1,2n}(x,u,v)=\delta(x)P_1(u,v)
\end{eqnarray*}
for some $P_1(u,v)\in \mathcal{H}_1\otimes \mathcal{H}_1^*$. Then we have
\begin{eqnarray*}
&&\int_{\mathbb{S}^{m-1}}\mathcal{D}_{1,2n}\overline{E_{1,2n}(x,u,v)}Q_1(v)dS(v)\\
&=&\delta(x)\int_{\mathbb{S}^{m-1}}\overline{P_1(u,v)}Q_1(v)dS(v).
\end{eqnarray*}
Now, for all $Q_1\in\mathcal{H}_1$, we have
\begin{eqnarray*}
&&\int_{\mathbb{S}^{m-1}}\mathcal{D}_{1,2n}\overline{E_{1,2n}(x,u,v)}Q_1(v)dS(v)\\
&=&\mathcal{D}_{1,2n}\int_{\mathbb{S}^{m-1}}c_1||x||^{2n-m}\overline{Z_1(\frac{xux}{||x||^2},v)}Q_1(v)dS(v)\\
&=&\mathcal{D}_{1,2n}\int_{\mathbb{S}^{m-1}}c_1||x||^{2n-m}\overline{Z_1(\frac{xux}{||x||^2},\frac{xv'x}{||x||^2})}Q_1(\frac{xv'x}{||x||^2})dS(v'),
\end{eqnarray*}
where in the last line we made a change of variables in the second argument of $Z_1$. Since $Z_1(u,v)$ is invariant under reflection and $\displaystyle\frac{xux}{||x||^2}$ is a reflection of variable $u$ in the direction of $x$, the last line in the last equation becomes
\begin{eqnarray*}
&&\mathcal{D}_{1,2n}\int_{\mathbb{S}^{m-1}}c_1\overline{Z_1(u,v')}||x||^{2n-m}Q_1(\frac{xv'x}{||x||^2})dS(v')\\
&=&c_1\mathcal{D}_{1,2n}||x||^{2n-m}Q_1(\frac{xux}{||x||^2}).
\end{eqnarray*}
Hence, we obtain
\begin{eqnarray*}
\delta(x)\int_{\mathbb{S}^{m-1}}\overline{P_1(u,v)}Q_1(v)dS(v)
=c_1\mathcal{D}_{1,2n}||x||^{2n-m}Q_1(\frac{xux}{||x||^2}).
\end{eqnarray*}

As the reproducing kernel $Z_1(u,v)$ is invariant under the $Spin(m)$-representation
\newline
$H:\ f(u,v)\mapsto sf(su\tilde{s},sv\tilde{s})\tilde{s}$, the kernel $E_{1,2n}(x,u,v)$ is also $Spin(m)$-invariant:
\begin{eqnarray*}
sE_{1,2n}(sx\tilde{s},su\tilde{s},sv\tilde{s})\tilde{s}=E_{1,2n}(x,u,v).
\end{eqnarray*}

\noindent From this it follows that $P_1(u,v)$ must be also invariant under $H$. Let now $\phi$ be a test function with $\phi(0)=1$. Let $L$ be the action of $Spin(m)$ given by $L: f(u)\mapsto sf(\tilde{s}us)\tilde{s}.$
Then
\begin{eqnarray*}
&&\langle \mathcal{D}_{1,2n}\big(c_1||x||^{2n-m}L(\frac{x}{||x||})L(s)Q_1(u)\big),\phi(x)\rangle\\
&=&\int_{\mathbb{S}^{m-1}}\overline{P_1(u,v)}L(s)Q_1(v)dS(v)\\
&=&L(s)\int_{\mathbb{S}^{m-1}}\overline{P_1(u,v)}Q_1(v)dS(v)\\
&=&\langle L(s)\big(\mathcal{D}_{1,2n}c_1||x||^{2n-m}L(\frac{x}{||x||})Q_1(u)\big),\phi(x)\rangle.
\end{eqnarray*}
In this way we have constructed an element of $End(\mathcal{H}_1)$ commuting with the $L$-representation of $Spin(m)$ that is irreducible; see Section 2.2.2. By Schur's Lemma (\cite{F}), it follows that $P_1(u,v)$ must be the reproducing kernel $Z_1(u,v)$ if we choose $c_1$ properly. Hence
\begin{eqnarray*}
\mathcal{D}_{1,2n}E_{1,2n}(x,u,v)=\delta(x)Z_1(u,v).
\end{eqnarray*}
\par
\subsubsection*{$k$ odd,\ $k=2n-1$ (The fermionic case)}
The reproducing kernel  $Z_{k}(u,v)$ for degree $k$ homogeneous monogenic spherical polynomials, those in $\mathcal{M}_k$, is called the \emph{zonal spherical monogenic} \cite{B2}. (There should be no confusion using the same notation for zonal spherical harmonics and monogenics.) In our circumstance, for $u,v\in\mathbb{S}^{m-1}$,
$$Z_1(u,v)=\displaystyle\frac{1}{\omega_{m-1}}\bigg(\displaystyle\frac{2\mu+1}{2\mu}C^{\mu}_1(t)+(u\wedge v)C_0^{\mu+1}(t)\bigg),$$
where $\mu=\displaystyle\frac{m}{2}-1$, $t=\langle u,v\rangle$, $u\wedge v=uv+\langle u,v\rangle$, and $C^{\mu}_k(t)$ are the Gegenbauer polynomials \cite{B2}.
With similar arguments and the fact that $Z_1(u,v)$ is also $Spin(m)$-invariant under the same $Spin(m)$-action as in the even case, one can show that 
$$E_{1,2n-1}(x,u,v)=c'_1\frac{x}{||x||^{m-2n+2}}Z_1(\frac{xux}{||x||^2},v)$$
 is the fundamental solution of $\mathcal{D}_{1,2n-1}$, where $c'_1$ is a non-zero real constant.\\
 \par
Since $E_{1,k}(x,u,v)$ is the fundamental solution of $\DD$, we have
\begin{eqnarray*}
\int_{\Rm}\int_{\mathbb{S}^{m-1}}E_{1,k}(x-y,u,v)\DD \psi(x,u)dS(u)dx^m=\psi(y,v),
\end{eqnarray*}
where $\psi(x,u) \in C^{\infty}(\Rm, U)$ with compact support in $x$ for each $u\in \mathbb{R}^{m}$, $U=\Mo$ when $k$ is odd and $U=\Ho$ when $k$ is even. Hence, we have $\DD E_{1,k}=Id$ and $E_{1,k}=\DD^{-1}$ in the distribution sense. Now,
\begin{eqnarray*}
J_{-k}\mathcal{D}_{1,k,y,w} \psi(y,w)=\mathcal{D}_{1,k,x,u}J_k \psi(\phi(x),\frac{(cx+d)u(\widetilde{cx+d})}{||cx+d||^2}),
\end{eqnarray*}
where $y=\phi (x)=(ax+b)(cx+d)^{-1}$ is a M\"{o}bius transformation and $w=\displaystyle\frac{(cx+d)u(\widetilde{cx+d})}{||cx+d||^2}$ as in Theorem 3, we get
$$J_k^{-1}\mathcal{D}_{1,k,x,u}^{-1}J_{-k}=\mathcal{D}_{1,k,y,w}^{-1}.$$
Alternatively,
$$J_k^{-1}E_{1,k,x,u}J_{-k}=E_{1,k,y,w}.$$
This gives us the intertwiners of the fundamental solution $E_{1,k}$ under M\"{o}bious transformations, which also reveals that the fundamental solutions are conformally invariant under M\"{o}bius transformations.

\section{\textbf{Ellipticity of the operator $\mathcal{D}_{1,k}$}}

\par Notice that the bases of the target space $\Ho$ and $\mathcal{M}_1$ have simple expressions. We can use techniques similar to those in \cite{B1,E} to show that the operators $\mathcal{D}_{1,k}$ are elliptic. First, we introduce the definition for an elliptic operator.
\begin{definition}
A linear homogeneous differential operator of $k$-th order $\mathcal{D}_{1,k}: C^{\infty}(\mathbb{R}^m, V_1)\longrightarrow C^{\infty}(\mathbb{R}^m, V_2)$ is elliptic if for every non-zero vector $x\in \mathbb{R}^m$ its principal symbol, the linear map $\sigma_x(\mathcal{D}_{1,k}):V_1\longrightarrow V_2$ obtained by replacing its partial derivatives $\partial_{x_j}$ with the corresponding variables $x_j$, is a linear isomorphism.
\end{definition}
Then we prove ellipticity of $\DD$ in the even and odd cases individually.

\subsubsection*{ $k$ even, $k=2n$ (The bosonic case) }
\begin{theorem}\label{even elliptic}
The operator $\mathcal{D}_{1,2n}:=\bigg(\triangle_x-\frac{4n}{m+2n-2}\langle u,D_x\rangle \langle D_u,D_x\rangle \bigg)\Delta_x^{n-1}$ is an elliptic operator.
\end{theorem}
\begin{proof}
In \cite{E} it was shown that the operator $\triangle_x-\frac{4}{m}\langle u,D_x\rangle \langle D_u,D_x\rangle $ is elliptic. In our case, the term in the parentheses is the same as the previous one up to a constant coefficient, so a similar argument shows $\triangle_x-\frac{4n}{m+2n-2}\langle u,D_x\rangle \langle D_u,D_x\rangle $ is elliptic. Since the symbol of $\Delta_x^{n-1}$ is non-negative, $\bigg(\triangle_x-\frac{4n}{m+2n-2}\langle u,D_x\rangle \langle D_u,D_x\rangle \bigg)\Delta_x^{n-1}$ is elliptic.
\end{proof}

\subsubsection*{$k$ odd, $k=2n-1$ (The fermionic case)}
\begin{theorem}\label{odd elliptic}
The operator
\begin{eqnarray*}
\mathcal{D}_{1,2n-1}:=D_x\Delta_x^{n-1}-\frac{2}{m+2n-2}u\langle D_u,D_x\rangle \Delta_x^{n-1}-\frac{4n-4}{m+2n-2}\langle u,D_x\rangle \langle D_u,D_x\rangle \Delta_x^{n-2}D_x
\end{eqnarray*}
 is an elliptic operator.
\end{theorem}
\begin{proof}
To prove the theorem, we show that, for fixed $x\in \mathbb{R}^m,$ the symbol of the operator $\mathcal{D}_{1,2n-1}$, which is given by
\begin{eqnarray*}
x||x||^{2n-2}-\frac{2}{m+2n-2}u\langle D_u,x\rangle ||x||^{2n-2}-\frac{4n-4}{m+2n-2}\langle u,x\rangle \langle D_u,x\rangle ||x||^{2n-4}x,
\end{eqnarray*}
is a linear isomorphism from $\Mo$ to $\Mo$. As the symbol is clearly a linear map, it remains to be proven that the map is injective. Recall that $\Mo$ is actually $\Mo(\Rm,\mathcal{S})$, however, if we can prove the symbol is injective for $\Mo(\mathcal{C}l_m)$, then this also implies that it is injective for $\Mo({\mathcal{S}})\subset\Mo(\mathcal{C}l_m(\C))$. From the Almansi-Fischer decomposition $\mathcal{H}_1=\Mo\oplus u\mathcal{M}_0$, it is easy to obtain that $dim \Mo=m-1$. Since $\{e_ju_m+e_mu_j\}_{j=1}^{m-1}$ are in $\Mo$ and it is also a linearly independent set in $\Mo$. Therefore, it is actually a basis of $\Mo$. Hence, an arbitrary element of $\mathcal{M}_1$ can be written as $\sum_{j=1}^m\alpha_j(e_ju_m+e_mu_j)$ with $\alpha_j\in \mathbb{C}$ for all $1 \leq j\leq m$. We next show that the following system of equations has a unique solution:
\begin{eqnarray*}
\big(x||x||^{2}-\frac{2u\langle D_u,x\rangle ||x||^{2}}{m+2n-2}-\frac{4n-4}{m+2n-2}x\langle u,x\rangle \langle D_u,x\rangle\big)\big(\sum_{j=1}^m\alpha_j(e_ju_m+e_mu_j)\big)=0.
\end{eqnarray*}
With $c_1=\frac{2}{m+2n-2}$, $c_2=\frac{4n-4}{m+2n-2}$, $a_i=(c_1e_i||x||^2+c_2xx_i)$, $b_j=x_me_j+x_je_m$, and $1 \leq i,j\leq m-1$, this equation system can be written in matrix notation as follows:
\begin{equation*}
          \begin{bmatrix}
          -x||x||^2e_m-a_1b_1 & -a_1b_2 & \ldots & -a_1b_{m-1} \\
          -a_2b_1  &    -x||x||^2e_m-a_2b_2  & \ldots &  -a_2b_{m-1} \\
          \vdots & \vdots & \ddots & \vdots \\
         -a_{m-1}b_1 & -a_{m-1}b_2 & \ldots & -x||x||^2e_m-a_{m-1}b_{m-1}
          \end{bmatrix}
          \begin{bmatrix}
          \alpha_1 \\
          \alpha_2 \\
          \vdots \\
          \alpha_{m-1}
          \end{bmatrix}
          =0.
\end{equation*}
In order to show that this system has an unique solution, it suffices to prove that
\begin{equation*}
           \begin{vmatrix}
            -x||x||^2e_m-a_1b_1 & -a_1b_2 & \ldots & -a_1b_{m-1} \\
          -a_2b_1  &    -x||x||^2e_m-a_2b_2  & \ldots &  -a_2b_{m-1} \\
          \vdots & \vdots & \ddots & \vdots \\
         -a_{m-1}b_1 & -a_{m-1}b_2 & \ldots & -x||x||^2e_m-a_{m-1}b_{m-1}
           \end{vmatrix}
           \neq 0.
\end{equation*}
Using the notation $\vec{a}=(a_1,a_2,\dots,a_{m-1})^{T}$ and $\vec{b}=(b_1,b_2,\dots,b_{m-1})^{T}$, the determinant can be written more compactly as
\begin{eqnarray*}
P(x)=det\big(-x||x||^2e_m\mathbb{I}_{m-1}-\vec{a}\cdot\vec{b}^T\big)\neq 0.
\end{eqnarray*}
As a function,
\begin{eqnarray*}
P(x)&=&det\big(-x||x||^2e_m\mathbb{I}_{m-1}-\vec{a}\cdot\vec{b}^T\big)=det\big(-x||x||^2e_m\mathbb{I}_{m-1}-\vec{a}\cdot\vec{b}^T\big)^T \\
&=&-x||x||^2e_m-\vec{a}^T\cdot\vec{b}
=-x||x||^2e_m-\sum_{j=1}^{m-1}(c_1e_j||x||^2+c_2xx_j)(x_me_j+x_je_m) \\
&=&-x||x||^2e_m+(c_1+c_2)||x||^2x_m-(c_1+c_2)x||x||^2e_m
\end{eqnarray*}
Checking each $e_j$-th component with $1\leq  j\leq m$, it is easy to see $P(x)$ is non-zero if $x$ is non-zero. This completes the proof.
\end{proof}


\end{document}